\documentclass{article}
\usepackage{amsmath,amssymb,amsthm}
\usepackage{pstricks,pst-node}
\usepackage{graphicx}

\newtheorem{thm}{Theorem}[section]
\newtheorem{cor}[thm]{Corollary}
\newtheorem{defn}[thm]{Definition}
\newtheorem{lem}[thm]{Lemma}
\newtheorem{obs}[thm]{Observation}

\numberwithin{equation}{section}

\newcommand{\F}{\mathcal{F}}
\newcommand{\G}{\mathcal{G}}

\title{Hereditary unigraphs and Erd\H{o}s--Gallai equalities}
\author{Michael D. Barrus\\Department of Mathematics, Brigham Young University\\Provo, UT 84602, USA.\\Phone: (801) 422-2336. FAX: (801) 422-0504.\\\texttt{barrus@math.byu.edu}}

\begin{document}

\maketitle

\begin{abstract}
We give characterizations of the structure and degree sequences of hereditary unigraphs, those graphs for which every induced subgraph is the unique realization of its degree sequence. The class of hereditary unigraphs properly contains the threshold and matrogenic graphs, and the characterizations presented here naturally generalize those known for these other classes of graphs.

The degree sequence characterization of hereditary unigraphs makes use of the list of values $k$ for which the $k$th Erd\H{o}s--Gallai inequality holds with equality for a graphic sequence. Using the canonical decomposition of Tyshkevich, we show how this list describes structure common among all realizations of an arbitrary graphic sequence.\medskip

\end{abstract}

\section{Introduction}\label{sec: one}

Unigraphs are those graphs such as $C_5$ or $3K_2$ with the remarkable property that they are the unique graphs (up to isomorphism) having their respective degree sequences. Relatively few graphs satisfy this requirement, but those that do comprise a number of interesting classes. For example, edgeless graphs and complete graphs are certainly unigraphs. These trivial examples are included in the class of threshold graphs, which were defined by Chv\'{a}tal and Hammer~\cite{ChvatalHammer73} in connection with set-packing problems and independently discovered by several other authors in varying contexts (see Chapter 1 in~\cite{MahadevPeled95} for a summary). Matroidal graphs and matrogenic graphs were introduced by Peled~\cite{Peled77} and F\"{o}ldes and Hammer~\cite{FoldesHammer78}, respectively, as graphs for which certain edge or vertex subsets form the circuits of a matroid. These classes satisfy the following inclusions: \begin{equation} \label{eq: tower}
\left.\begin{aligned}
  \text{complete}\\
  \text{edgeless}
\end{aligned}\right\rbrace \subset
\text{threshold} \subset \text{matroidal} \subset \text{matrogenic} \subset \text{unigraph}.\end{equation} Graphs in these classes have been well studied; for surveys of results, see~\cite{BrandstadtEtAl99} and~\cite{MahadevPeled95}.

One interesting thing to note is that the proper subfamilies of unigraphs in \eqref{eq: tower} are all hereditary (i.e., closed under taking induced subgraphs), while the class of unigraphs itself is not. For example, the unique graph with degree sequence $(4,2,2,2,2,2)$ contains two nonisomorphic induced subgraphs with degree sequence $(3,2,2,2,1)$. Motivated by this contrast, we might ask how far a hereditary class may enlarge the class of matrogenic graphs while staying within the family of unigraphs. The class of \emph{hereditary unigraphs}, defined in~\cite{Barrus12}, is the maximal hereditary class containing only unigraphs; it may be thought of as the union of all hereditary families of unigraphs.

In this paper we show that the hereditary unigraphs generalize important properties possessed by graphs in the hereditary classes in~\eqref{eq: tower}. In particular, these latter classes all have three specific types of characterizations, in terms of forbidden induced subgraphs, strict requirements on vertex adjacencies, and degree sequences. We will see that these same types of characterizations exist for hereditary unigraphs.

Every hereditary family is characterized by a list of forbidden induced subgraphs, and this list is known for each of the hereditary families in~\eqref{eq: tower}. Complete graphs and edgeless graphs forbid $2K_1$ and $K_2$, respectively. Threshold graphs~\cite{ChvatalHammer73} are precisely the $\{2K_2,C_4,P_4\}$-free graphs. Matrogenic graphs~\cite{FoldesHammer78} and matroidal graphs~\cite{Peled77} have characterizations in terms of ten and eleven forbidden subgraphs, respectively, each having five vertices. In~\cite{Barrus12} the author provided a list of forbidden subgraphs for the hereditary unigraphs. (Notation and definitions will be given later in the paper.)

\begin{thm}[\cite{Barrus12}] \label{thm: Barrus uni}
A graph is a hereditary unigraph if and only if it contains no element of \begin{multline}\nonumber \{P_5, \overline{P_5}, K_2+K_3, K_{2,3}, \text{4-pan}, \text{co-4-pan}, 2P_3, \overline{2P_3},\\ K_2+P_4, \overline{K_2+P_4}, K_2+C_4, \overline{K_2+C_4}, R, \overline{R}, S, \overline{S}\}\end{multline} as an induced subgraph.
\end{thm}

Hereditary unigraphs also have characterizations in terms of their structure and degree sequences. We present these in Theorems~\ref{thm: main structural}, \ref{thm: deg seq for her unis, part I}, and~\ref{thm: deg seq for her unis, part II} and show how they are natural generalizations of the characterizations for threshold and matrogenic graphs, which we recall later.

Our structural results will begin with the class of graphs containing no induced subgraph in $\{2K_2,C_4,\text{chair}\}$, where the chair (also known as the fork) is the tree with degree sequence $(3,2,1,1,1)$. This class properly contains the threshold graphs and was shown in~\cite{NonMinimalTriples} to contain only unigraphs. We show that these graphs and their complements may be assembled from building blocks known as spiders by ``expanding'' vertices according to certain rules. In relaxing these rules to allow more varied expansions while maintaining a class with a degree sequence characterization, we arrive at the family of hereditary unigraphs.

Graph families containing only unigraphs necessarily have degree sequence characterizations; to recognize membership of a graph $G$ in such a family $\mathcal{F}$, one could simply check for the degree sequence of $G$ in a (potentially infinitely long) list of the degree sequences of graphs in $\mathcal{F}$. However, the structural conditions satisfied by graphs in the families in~\eqref{eq: tower} lead to much more satisfying characterizations of their degree sequences. The same will be true in our description of the degree sequences of hereditary unigraphs.

Interestingly, our characterization is closely related to the well known Erd\H{o}s--Gallai inequalities~\cite{ErdosGallai60} for determining when a list of nonnegative integers is the degree sequence of a graph. Given a degree sequence $(d_1,\dots,d_n)$ with its terms in descending order and a nonnegative integer $k \leq n$, the \emph{$k$th Erd\H{o}s--Gallai inequality} is \[\sum_{i=1}^k d_i \leq k(k-1) + \sum_{i=k+1}^n \min\{k,d_i\}.\] In order for a list of nonnegative integers with even sum to be the degree sequence of a graph, it must satisfy each of the Erd\H{o}s--Gallai inequalities. Threshold graphs and split graphs have degree sequence characterizations requiring one or more of the inequalities to further hold with equality (\cite{HammerEtAl78} and~\cite{HammerSimeone81}; see Section~\ref{sec: five}). In characterizing the degree sequences of hereditary unigraphs we study the list of integers $k$ for which the $k$th inequality holds with equality. Using the canonical decomposition of Tyshkevich~\cite{Tyshkevich00}, we describe what this list reveals about the structure of an arbitrary graph having a specified degree sequence.

We proceed as follows: In Sections 2 through 4 we describe the structure of hereditary unigraphs. In Section 2 we recall the canonical decomposition of a graph as described by Tyshevich~\cite{Tyshkevich00} and then characterize the structure of $\{2K_2,C_4,\text{chair}\}$-free graphs and their complements, showing how they may be obtained via vertex expansions in spiders. In Section 3 we explore these expansions further, arriving at a larger hereditary class with a degree sequence characterization. In Section 4 we relate this class to the hereditary unigraphs to obtain a structural characterization of the latter. In Section 5 we use the canonical decomposition to study the equalities among the Erd\H{o}s--Gallai inequalities of a degree sequence. We conclude in Section 6 by characterizing the degree sequences of hereditary unigraphs.

All graphs considered in this paper will have no loops or multiple edges. We use $V(G)$ and $E(G)$ to denote the vertex set and edge set of a graph $G$. Given a vertex $u$ in $V(G)$, we denote the neighborhood of $u$ (that is, the set of vertices adjacent to $u$ in $G$) by $N_G(u)$. Given a subset $W$ of $V(G)$, we use $G[W]$ to denote the induced subgraph of $G$ with vertex set $W$. We adopt the convention that when listing the degree sequence $(d_1,\dots,d_n)$ of a graph, the terms will always appear in descending order, and in place of $\alpha$ copies of a sequence term $k$, we may write $k^{\alpha}$. We denote the complement of $G$ by $\overline{G}$. The disjoint union of graphs $G$ and $H$ will be denoted by $G+H$, and the disjoint union of $m$ copies of a graph $G$ will be indicated by $mG$. The join of $G$ and $H$ will be denoted by $G \vee H$. Complete graphs, cycles, and paths with $n$ vertices will be denoted respectively by $K_n$, $C_n$, and $P_n$. The complete bigraph with partite sets of sizes $m$ and $n$ will be denoted $K_{m,n}$. The 4-pan is defined as the graph obtained by attaching a vertex of degree 1 to a 4-cycle, and the co-4-pan is the complement of the 4-pan. The graphs $R$, $\overline{R}$, $S$, $\overline{S}$ mentioned in Theorem~\ref{thm: Barrus uni} will be defined in Section~\ref{sec: three} and illustrated in Figures~\ref{fig: R,R-bar} and~\ref{fig: other forb subgr}.

\section{$\{2K_2,C_4,\text{chair}\}$-free graphs and their complements} \label{sec: two}
Given a set $\mathcal{F}$ of graphs, we say a graph $G$ is $\mathcal{F}$-free if $G$ contains no induced subgraph isomorphic to an element of $\mathcal{F}$. In this section we describe the structure of $\{2K_2,C_4,\text{chair}\}$-free graphs and $\{2K_2,C_4,\text{kite}\}$-free graphs, where the chair is the tree shown in Figure~\ref{fig: chairkite intro} together with its complement, the kite. In a sense to be made more precise in later sections, graphs in these classes are prototypical hereditary unigraphs.

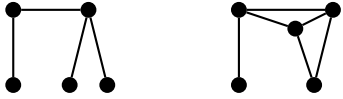
\begin{figure}
\centering
\begin{pspicture}(10,0)(14.25,1)
\cnode*(10,0){3pt}{a} \cnode*(10,1){3pt}{e} \cnode*(11,1){3pt}{d} \cnode*(10.75,0){3pt}{b} \cnode*(11.25,0){3pt}{c}
\ncline{-}{a}{e} \ncline{-}{e}{d} \ncline{-}{d}{b} \ncline{-}{d}{c}
\cnode*(13,0){3pt}{v} \cnode*(13,1){3pt}{z} \cnode*(14,0){3pt}{w} \cnode*(13.75,0.75){3pt}{y} \cnode*(14.25,1){3pt}{x}
\ncline{-}{v}{z} \ncline{-}{z}{y} \ncline{-}{z}{x} \ncline{-}{y}{x} \ncline{-}{y}{w} \ncline{-}{x}{w}
\end{pspicture}
\caption{The chair and kite.}
\label{fig: chairkite intro}
\end{figure}

We begin by recalling some relevant results.

\begin{thm}[Bl\'{a}zsik et al.~\cite{BlazsikEtAl93}]\label{thm: pseudo-split char}
A graph $G$ is $\{2K_2,C_4\}$-free if and only if its vertex set may be partitioned into sets $A$, $B$, and $C$ such that
\begin{enumerate}
\item[\textup{(i)}] $A$ is an independent set and $B$ is a clique;
\item[\textup{(ii)}] $C$ is empty, or $G[C] \cong C_5$;
\item[\textup{(iii)}] each vertex in $C$ is adjacent to every vertex in $B$ and to no vertex in $A$.
\end{enumerate}
\end{thm}

A graph is \emph{split} if its vertex set may be partitioned into a clique and an independent set; in the $\{2K_2,C_4\}$-free graph in the last result, $G[A \cup B]$ is a split graph. Indeed, we have the following.

\begin{thm}[F\"{o}ldes and Hammer~\cite{FoldesHammer76}] \label{thm: split forb sub}
A graph is split if and only if it is $\{2K_2,C_4,C_5\}$-free.
\end{thm}

Because of the similarity of their vertex set partitions and forbidden subgraph characterization to those possessed by the split graphs, $\{2K_2,C_4\}$-free graphs are known as \emph{pseudo-split graphs}.

The decomposition presented in Theorem~\ref{thm: pseudo-split char} is generalized and refined considerably by the canonical decomposition of R. I. Tyshkevich. This decomposition provides a unifying framework for the characterizations of $\{2K_2,C_4,\text{chair}\}$-free graphs, the hereditary families in~\eqref{eq: tower}, and the full family of hereditary unigraphs. Our brief presentation here follows~\cite{Tyshkevich00}.

Given a split graph with a partition of $V(G)$ into an independent set $A$ and a clique $B$, we call the triple $(G,A,B)$ a \emph{splitted graph}. Two splitted graphs $(G,A,B)$ and $(G',A',B')$ are \emph{isomorphic} if there exists a graph isomorphism $\varphi:V(G)\to V(G')$ such that $\varphi(A) = A'$. Given a splitted graph $(G,A,B)$ and a graph $H$ on disjoint vertex sets, we define the \emph{composition} of $(G,A,B)$ and $H$ to be the graph $(G,A,B) \circ H$ formed by adding to $G+H$ all edges $uv$ such that $u \in B$ and $v \in V(H)$. For example, when $H = K_3$ and $G=P_4$, with $A$ the set of endpoints and $B$ the set of midpoints of $G$, the composition $(G,A,B) \circ H$ is the graph on the left in Figure~\ref{fig: composition} (here heavy lines denote all possible edges between vertex sets). On the right we show $(G, A, B) \circ ((G,A,B) \circ H)$. The composition of two splitted graphs may also be thought of as a splitted graph, where the independent set and clique are the unions of the independent sets and of the cliques, respectively. The operation $\circ$ is associative, so in the future we will omit grouping parentheses when performing multiple compositions. Observe that in a composition $(G_k,A_k,B_k) \circ \dots \circ (G_1,A_1,B_1)\circ G_0$, each vertex in $B_i$ is adjacent to every vertex in $\bigcup_{j<i} V(G_j)$, each vertex in $A_i$ is adjacent to none of the vertices in $\bigcup_{j<i} V(G_j)$, and only the rightmost graph in the composition can fail to be a split graph.

\begin{figure}
\centering
\begin{pspicture}(-0.25,-0.25)(12.07,1.25)
\cnode*(0,0){3pt}{A} \cnode*(0,1){3pt}{B} \cnode*(1,0){3pt}{C} \cnode*(1,1){3pt}{D}
\cnode*(2.5,0.17){3pt}{E} \cnode*(2.88,0.83){3pt}{G} \cnode*(3.26,0.17){3pt}{F}
\ncline{-}{A}{B} \ncline{-}{B}{D} \ncline{-}{C}{D} \ncline{-}{E}{F} \ncline{-}{E}{G} \ncline{-}{F}{G}
\psellipse[linestyle=dashed](0.5,0)(0.75,0.25) \psellipse[linestyle=dashed](0.5,1)(0.75,0.25)
\pscircle[linestyle=dashed](2.88,0.39){0.64}
\uput[l](-0.25,0){\fontsize{7pt}{7pt}$A$} \uput[l](-0.25,1){\fontsize{7pt}{7pt}$B$} \uput[r](3.52,0.39){\fontsize{7pt}{7pt}$V(H)$}
\pnode(1.09,0.85){p1} \pnode(2.26,0.55){p2}
\ncline[linewidth=3pt]{p1}{p2}
\cnode*(6,0){3pt}{AA} \cnode*(6,1){3pt}{BB} \cnode*(7,0){3pt}{CC} \cnode*(7,1){3pt}{DD}
\cnode*(8.5,0){3pt}{AAA} \cnode*(8.5,1){3pt}{BBB} \cnode*(9.5,0){3pt}{CCC} \cnode*(9.5,1){3pt}{DDD}
\cnode*(11,0.17){3pt}{EE} \cnode*(11.38,0.83){3pt}{GG} \cnode*(11.76,0.17){3pt}{FF}
\ncline{-}{AA}{BB} \ncline{-}{BB}{DD} \ncline{-}{CC}{DD} \ncline{-}{AAA}{BBB} \ncline{-}{BBB}{DDD} \ncline{-}{CCC}{DDD} \ncline{-}{EE}{FF} \ncline{-}{EE}{GG} \ncline{-}{FF}{GG}
\psellipse[linestyle=dashed](6.5,0)(0.75,0.25) \psellipse[linestyle=dashed](6.5,1)(0.75,0.25)
\psellipse[linestyle=dashed](9,0)(0.75,0.25) \psellipse[linestyle=dashed](9,1)(0.75,0.25)
\pscircle[linestyle=dashed](11.38,0.39){0.64}
\pnode(7.2,0.91){p3} \pnode(10.74,0.47){p4}
\pnode(9.59,0.85){p5} \pnode(10.76,0.55){p6}
\pnode(6.98,0.81){p7} \pnode(8.52,0.19){p8}
\pnode(7.25,1){p9} \pnode(8.25,1){p10}
\ncline[linewidth=3pt]{p3}{p4}
\ncline[linewidth=3pt]{p5}{p6}
\ncline[linewidth=3pt]{p7}{p8}
\ncline[linewidth=3pt]{p9}{p10}
\end{pspicture}
\caption{The compositions $(G, A, B) \circ H$ and $(G,A,B)\circ(G,A,B)\circ H$.}
\label{fig: composition}
\end{figure}
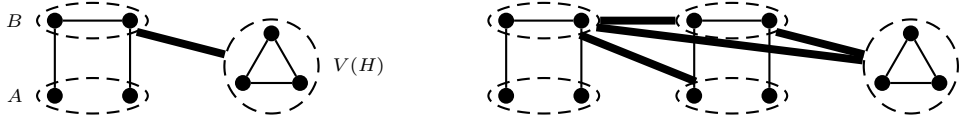

A graph is \emph{decomposable} if it can be written as a composition $(G,A,B) \circ H$ where $G$ and $H$ both have at least one vertex. Otherwise, it is \emph{indecomposable}. Maximal indecomposable induced subgraphs of a graph are called its \emph{indecomposable} (or \emph{canonical}) \emph{components}. Tyshkevich showed the following:

\begin{thm}[Tyshkevich~\cite{Tyshkevich80, Tyshkevich00}] \label{thm: canonical decomp is unique}
Every graph $G$ can be expressed as a composition \[G = (G_k, A_k, B_k) \circ \dots \circ (G_1, A_1, B_1) \circ G_0 \tag{$*$}\] of indecomposable components. Here the $(G_i, A_i, B_i)$ are indecomposable splitted graphs and $G_0$ is an indecomposable graph. (If $G$ is indecomposable, then $k=0$.) Graphs $G$ and $G'$ expressed as \textup{($*$)} and \[G' = (G'_\ell, A'_\ell, B'_\ell) \circ \dots \circ (G'_1, A'_1, B'_1) \circ G'_0\] are isomorphic if and only if $G_0 \cong G'_0$, $k = \ell$, and $(G_i, A_i, B_i) \cong (G'_i, A'_i, B'_i)$ for $1 \leq i \leq k$.
\end{thm}

Since \textup{($*$)} is the unique expression of $G$ as a composition of indecomposable components, up to isomorphism of the components, we call \textup{($*$)} the \emph{canonical decomposition} of $G$.

Note that if a graph $H$ is decomposable, then we may express it as $(H_1,A,B) \circ H_0$ for some induced subgraphs $H_1$ and $H_0$. Then $\overline{H}=(\overline{H_1},B,A) \circ \overline{H_0}$, so $\overline{H}$ is decomposable as well. This leads us to the following.

\begin{obs}\label{obs: complement decompositions}
If $G$ has canonical decomposition $(G_k, A_k, B_k) \circ \dots \circ (G_1, A_1, B_1) \circ G_0$, then the canonical decomposition of $\overline{G}$ is $(\overline{G_k}, B_k, A_k) \circ \dots \circ (\overline{G_1}, B_1, A_1) \circ \overline{G_0}$.
\end{obs}

In~\cite{A4structure} West and the author gave a useful tool for recognizing canonically indecomposable graphs.

\begin{thm}[\cite{A4structure}]\label{thm: A4 indecomp}
A graph $G$ is canonically indecomposable if and only if for every pair $u,v$ of distinct vertices in $G$, there exists a sequence of induced subgraphs $H_1,\dots,H_t$ of $G$ such that
\begin{enumerate}
\item[\textup{(i)}] $u$ is a vertex of $H_1$, and $v$ is a vertex of $H_t$;
\item[\textup{(ii)}] consecutive subgraphs in the sequence have at least one common vertex;
\item[\textup{(iii)}] each $H_i$ is isomorphic to $2K_2$, $C_4$, or $P_4$.
\end{enumerate}
\end{thm}

We now use the canonical decomposition in characterizing the $\{2K_2,C_4,\text{chair}\}$-free graphs. In light of the conditions of the canonical decomposition and Theorem~\ref{thm: pseudo-split char}, the following is immediate.

\begin{cor} \label{cor: split,pseudo-split chars}
Let $G$ be a graph with canonical decomposition $(G_k,A_k,B_k) \circ \dots \circ (G_1,A_1,B_1)\circ G_0$.
 \begin{enumerate}
 \item[\textup{(a)}] $G$ is split if and only if $G_0$ is split.
 \item[\textup{(b)}] $G$ is $\{2K_2,C_4\}$-free if and only if $G_0$ is split or $G_0$ is isomorphic to $C_5$.
\end{enumerate}
\end{cor}

Chair-free split graphs have been studied before; we present the result of Brandst\"{a}dt and Mosca~\cite{BrandstadtMosca}.

A \emph{module} $M$ in a graph $G$ is a subset of $V(G)$ with the property that every vertex not in $M$ is adjacent to either all or none of the vertices of $M$. We say that $G$ is \emph{prime} if each of its modules consists of a single vertex or is equal to $V(G)$.

\begin{defn}\label{defn: prime spider}
A \emph{(prime) spider} is a graph whose vertex set admits a partition into sets $A$, $B$, and $C$ such that
\begin{enumerate}
\item[\textup{(i)}] $A$ is an independent set, $B$ is a complete graph, and $|A|=|B| \geq 2$;
\item[\textup{(ii)}] $|C|\leq 1$, and a vertex in $C$ is adjacent to every vertex in $B$ and to no vertex in $A$;
\item[\textup{(iii)}] there is a bijection $f:A \to B$ such that either \textup{(a)} each vertex $v$ in $A$ is adjacent only to $f(v)$, or \textup{(b)} each vertex $v$ in $A$ to adjacent to all vertices of $B$ except $f(v)$.
\end{enumerate}
\end{defn}

We call vertices in $A$ the \emph{feet} and vertices in $B$ the \emph{body vertices} of the spider. The vertex in $C$, if it exists, is called the \emph{head}; if $C$ is empty, the spider is \emph{headless}.

\begin{thm}[\cite{BrandstadtMosca}] \label{thm: BrandstadtMosca}
If a chair-free split graph is prime and has more than one vertex, then it is a prime spider.
\end{thm}

The split canonical components of a $\{2K_2,C_4,\text{chair}\}$-free graph need not be prime, in general, so we extend this result. We first give some definitions.

Let $H$ and $J$ be graphs, and let $v$ be a vertex of $J$. To \emph{substitute $H$ for $v$} is to take the disjoint union of $H$ and $J-v$ and add all edges of the form $uw$ where $u \in V(H)$ and $w \in N_J(v)$. A \emph{top-expanded spider} is a graph that can be obtained by substituting complete graphs for body vertices in a headless spider. A \emph{bottom-expanded spider} is a graph obtained by substituting edgeless graphs for feet in a headless spider. Note that the complement of a bottom-expanded spider is a top-expanded spider.

A module is \emph{proper} if it does not contain all vertices of $G$. Given a graph $G$ and a module $M$, to \emph{contract} $M$ to a single vertex is to replace the vertices of $M$ by a vertex having the same neighbors among $V(G)\setminus M$ that vertices in $M$ had; alternatively, we may delete all but one of the vertices of $M$ from $G$.

\begin{thm} \label{thm: chair-free structure}
A graph $G$ with canonical decomposition $G=(G_k,A_k,B_k) \circ \dots \circ (G_1,A_1,B_1) \circ G_0$ is $\{2K_2,C_4,\text{chair}\}$-free if and only if $G_0$ is split or $C_5$, and each split indecomposable component having more than one vertex is a top-expanded spider.
\end{thm}
\begin{proof}
Let $G$ be a graph with canonical decomposition $(G_k,A_k,B_k) \circ \dots \circ (G_1,A_1,B_1)\circ G_0$. By Corollary~\ref{cor: split,pseudo-split chars}, $G$ is $\{2K_2,C_4\}$-free if and only if $G_0$ is isomorphic to $C_5$ or split. By Theorem~\ref{thm: A4 indecomp}, the chair is indecomposable. It follows that $G$ is $\{2K_2,C_4,\text{chair}\}$-free if and only if $G_0$ is isomorphic to $C_5$ or split and each canonical component $G_i$ is chair-free. Observing that $K_1$ and $C_5$ are chair-free, as is any top-expanded spider, it suffices to show that each split $G_i$ having more than one vertex is a top-expanded spider if it is $\{2K_2,C_4,\text{chair}\}$-free.

Henceforth assume that $G_i$ is a chair-free split canonical component of $G$ having more than one vertex. Since $P_4$ is self-complementary, it follows from Theorems~\ref{thm: split forb sub} and~\ref{thm: A4 indecomp} that both $G_i$ and $\overline{G_i}$ are connected. By the modular decomposition theorem of Gallai~\cite{Gallai67} the maximal proper modules of $G_i$ are disjoint. Contracting each maximal proper module to a single vertex yields a prime graph $G_i^*$ that is split and chair-free and has more than one vertex. By Theorem~\ref{thm: BrandstadtMosca}, $G^*$ is a prime spider. Let $A,B,C$ be a partition of $V(G^*)$ as in Definition~\ref{defn: prime spider}.

Note that $G_i$ may be obtained by substituting suitable graphs for vertices of $G_i^*$. However, each vertex $a$ of $A$ is the endpoint of an induced $P_4$ in $G_i^*$; substituting a graph with two or more vertices for $a$ creates either a chair or $2K_2$, neither of which is induced in $G_i$. Since each vertex $b$ in $B$ is a midpoint for an induced $P_4$ in $G_i$, substituting a non-complete graph for $b$ creates an induced $C_4$, which $G_i$ does not contain. Finally, note that if $C$ contains a vertex $c$, then $G_i=(G_i[A'\cup B'],A',B')\circ G'$, where $A'$ and $B'$ are the sets of vertices whose contracted versions respectively belong to $A$ and to $B$ in $G_i^*$, and $G'$ is the induced subgraph of $G_i$ whose vertices were contracted to the vertex $c$ in forming $G_i^*$. Since $G_i$ is indecomposable, this is a contradiction unless $C=\emptyset$. Thus $G_i$ is a top-expanded spider, as claimed.
\end{proof}

Since $\{2K_2,C_4\}$ and $\{\text{chair}, \text{kite}\}$ are pairs of complementary graphs, the $\{2K_2,C_4,\text{kite}\}$-free graphs are precisely the complements of the $\{2K_2,C_4,\text{chair}\}$-free graphs. Since $C_5$ is also self-complementary, by Observation~\ref{obs: complement decompositions} and Corollary~\ref{cor: split,pseudo-split chars} the indecomposable kite-free split graphs are precisely the complements of the indecomposable chair-free split graphs.
\begin{cor} \label{cor: kite-free structure}
A graph $G$ with canonical decomposition $G=(G_k,A_k,B_k) \circ \dots \circ (G_1,A_1,B_1) \circ G_0$ is $\{2K_2,C_4,\text{kite}\}$-free if and only if $G_0$ is split or $C_5$, and every split indecomposable component is a bottom-expanded spider.
\end{cor}

\section{More general vertex expansions} \label{sec: three}
In this section we introduce graph families generated by relaxing some of the conditions in the results from the previous section. One family in particular will be useful in studying the structure of hereditary unigraphs in Section~\ref{sec: four}.

By Theorem~\ref{thm: chair-free structure} and Corollary~\ref{cor: kite-free structure}, pseudo-split graphs that forbid both the chair and the kite have canonical components that are each one of $C_5$, a single vertex, or a headless spider. Removing the chair from the list of forbidden subgraphs permits the expansion (into cliques) of body vertices in the canonical components that are spiders; if we instead remove the kite from the list of forbidden subgraphs, we may expand feet vertices of the spiders into independent sets. Removing \emph{both} the chair and the kite from the list results in the complete class of pseudo-split graphs, including graphs that cannot be created by expanding vertices in spiders. Let $\G$ be the class of graphs that may be obtained from $\{2K_2,C_4,\text{chair},\text{kite}\}$-free graphs when both types of expansion are simultaneously allowed, that is, when complete graphs are substituted for body vertices and edgeless graphs are substituted for feet vertices in arbitrary spider canonical components.

\begin{figure}
\centering
\begin{pspicture}(9.25,-0.6)(14,1)
\cnode*(9.25,1){3pt}{W} \cnode*(10.25,1){3pt}{Z} \cnode*(11.25,1){3pt}{AA} \cnode*(9.25,0){3pt}{BB} \cnode*(10.25,0){3pt}{CC} \cnode*(11.25,0){3pt}{DD}
\ncline{-}{W}{Z} \ncline{-}{Z}{AA} \ncline{-}{Z}{CC} \ncline{-}{Z}{DD} \ncline{-}{AA}{DD} \ncline{-}{BB}{CC} \ncline{-}{CC}{DD}
\uput[d](10.25,-0.125){\fontsize{7pt}{7pt}$R$}
\cnode*(12,1){3pt}{EE} \cnode*(13,1){3pt}{FF} \cnode*(14,1){3pt}{GG} \cnode*(12,0){3pt}{HH} \cnode*(13,0){3pt}{II} \cnode*(14,0){3pt}{JJ}
\ncline{-}{EE}{FF} \ncline{-}{EE}{II} \ncline{-}{FF}{GG} \ncline{-}{FF}{II} \ncline{-}{FF}{JJ} \ncline{-}{GG}{JJ} \ncline{-}{HH}{II} \ncline{-}{HH}{JJ}
\uput[d](13,-0.05){\fontsize{7pt}{7pt}$\overline{R}$}
\end{pspicture}
\caption{The graphs $R$ and $\overline{R}$.}
\label{fig: R,R-bar}
\end{figure}
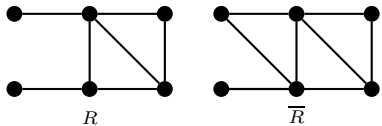
\begin{thm} \label{thm: R,R-bar}
The class $\G$ is precisely the set of $\{2K_2,C_4,R, \overline{R}\}$-free graphs, where $R$ and $\overline{R}$ are the graphs shown in Figure~\ref{fig: R,R-bar}.
\end{thm}
\begin{proof}
Let $G$ be a graph with canonical decomposition $G_k(A_k,B_k) \circ \dots \circ G_1(A_1,B_1) \circ G_0$.

Suppose first that $G_0$ is split or isomorphic to $C_5$, and that each indecomposable split component $G_i$ is isomorphic to $K_1$ or can be obtained from a headless spider by substituting complete graphs for body vertices and edgeless graphs for feet vertices. Neither $R$ nor $\overline{R}$ has this form; it follows that neither can be induced in $G_i$. Since $R$ and $\overline{R}$ are also indecomposable, we conclude that $G$ is $\{2K_2,C_4,R,\overline{R}\}$-free.

Suppose instead that $G$ is $\{2K_2,C_4,R,\overline{R}\}$-free. By Corollary~\ref{cor: split,pseudo-split chars}, $G_0$ is split or isomorphic to $C_5$. Let $(G_i, A_i, B_i)$ be a splitted canonical component such that $G_i \ncong K_1$. Consider the modules of $G_i$ that are wholly contained in $A_i$ or in $B_i$. Form $G'_i$ by contracting to single vertices the maximal elements of this collection under subset inclusion. Then $G'_i$ is a split graph having a partition $A'_i, B'_i$ into an independent set and a clique, where these are the sets resulting from $A_i$ and $B_i$ after the modules are contracted. Suppose that $G'_i$ induces the chair on vertices $a,b,c,d,e$, with $a,b,c \in A'_i$ and $d,e \in B'_i$, with $e$ adjacent to $a$, and $d$ adjacent to $b$ and $c$. Because of the module contractions, $\{b,c\}$ is not a module in $G'_i$. Assume $c$ is adjacent to a vertex $x$ in $B'_i$ to which $b$ is not. Recalling that $A'_i$ is an independent set and $B'_i$ is a clique, we find that $G'_i[\{a,b,c,d,e,x\}]$ is isomorphic to $R$ or $\overline{R}$, depending on whether $x$ is adjacent to $a$. This is a contradiction, since $G'_i$ is an induced subgraph of $G_i$, which was assumed to be $\{R,\overline{R}\}$-free. Thus $G'_i$ is chair-free.

Since $\{2K_2,C_4,R,\overline{R}\}$ is self-complementary, and $\overline{G_i}$ is also a split indecomposable graph, when we contract the modules in its independent set $B_i$ and clique $A_i$, we get the graph $\overline{G'_i}$. This graph is also chair-free by the argument above, so $G'_i$ is kite-free.

Since $G_i$ is indecomposable and has at least two vertices, it follows that $G'_i$ is also indecomposable and has at least two vertices. Since $G'_i$ is also $\{2K_2,C_4,\text{chair},\text{kite}\}$-free, from Theorem~\ref{thm: chair-free structure} and Corollary~\ref{cor: kite-free structure} we conclude that $G'_i$ is isomorphic to a headless spider. Since the modules contracted to form $G'_i$ were subsets of $A_i$ or $B_i$, the graph $G_i$ may be obtained from $G'_i$ by substituting edgeless graphs for vertices in $A'_i$ and complete graphs for vertices in $B'_i$, as claimed.
\end{proof}

By Theorem~\ref{thm: Barrus uni} and the comments in Section~\ref{sec: one}, the $\{2K_2,C_4,\text{chair}\}$-free graphs and $\{2K_2,C_4,\text{kite}\}$-free graphs are unigraphs, so they have degree sequence characterizations. One might wonder whether the $\{2K_2,C_4,R, \overline{R}\}$-free graphs can also be recognized from their degree sequences. This is not the case: the graph $R$, which does not belong to the class, shares its degree sequence with the graph $S$ shown in Figure~\ref{fig: other forb subgr}, which does belong to the class. Instead, we might ask which degree sequences are \emph{forcibly $\{2K_2,C_4,R,\overline{R}\}$-free}, that is, having only realizations that are $\{2K_2,C_4,R,\overline{R}\}$-free, and which graphs have these degree sequences.

\begin{figure}
\centering
\begin{pspicture}(9,-0.6)(14,1)
\cnode*(9,0){3pt}{AA} \cnode*(9,1){3pt}{AA2} \cnode*(9.67,0.5){3pt}{BB} \cnode*(10.33,1){3pt}{CC} \cnode*(10.33,0){3pt}{DD} \cnode*(11,0.5){3pt}{ZZ}
\ncline{-}{AA}{BB} \ncline{-}{AA2}{BB} \ncline{-}{BB}{CC} \ncline{-}{BB}{DD} \ncline{-}{CC}{DD} \ncline{-}{CC}{ZZ} \ncline{-}{DD}{ZZ}
\uput[d](10,-0.05){\fontsize{7pt}{7pt}$S$}
\cnode*(12,0.5){3pt}{EE} \cnode*(12.67,0.5){3pt}{FF} \cnode*(13.33,1){3pt}{GG} \cnode*(13.33,0){3pt}{HH} \cnode*(14,0){3pt}{II} \cnode*(14,1){3pt}{II2}
\ncline{-}{EE}{FF} \ncline{-}{FF}{GG} \ncline{-}{FF}{HH} \ncline{-}{GG}{HH} \ncline{-}{GG}{II} \ncline{-}{HH}{II} \ncline{-}{GG}{II2} \ncline{-}{HH}{II2}
\uput[d](13,-0.05){\fontsize{7pt}{7pt}$\overline{S}$}
\end{pspicture}
\caption{The graphs $S$ and $\overline{S}$.}
\label{fig: other forb subgr}
\end{figure}
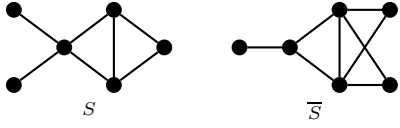
Alternatively, in the interest of finding a graph class with a degree sequence characterization that contains both the $\{2K_2,C_4,\text{chair}\}$-free graphs and $\{2K_2,C_4,\text{kite}\}$-free graphs, we could look at the graphs whose canonical components are each either $\{2K_2,C_4,\text{chair}\}$-free or $\{2K_2,C_4,\text{kite}\}$-free (though which is the case may differ from component to component). As we will see in Sections~\ref{sec: five} and~\ref{sec: six}, properties of the canonical decomposition imply that these graphs are recognizable from their degree sequences.

As it happens, these questions have the same answer, and the graph class involved is just what we need to characterize hereditary unigraphs in the next section.

\begin{thm}\label{thm: forcibly expandable}
The following are equivalent for a graph $G$:
\begin{enumerate}
\item[\textup{(a)}] $d(G)$ is forcibly $\{2K_2,C_4,R,\overline{R}\}$-free;
\item[\textup{(b)}] $G$ is pseudo-split, and every canonical component of $G$ is either chair-free or kite-free.
\item[\textup{(c)}] $G$ is $\{2K_2,C_4,R,\overline{R},S,\overline{S}\}$-free;
\end{enumerate}
\end{thm}

\begin{proof}
We prove that both (a) and (b) are equivalent to (c).

(a) $\Rightarrow$ (c): Since $d(G)$ is forcibly $\{2K_2,C_4, R, \overline{R}\}$-free, $G$ induces none of these four subgraphs. If $G$ induces $S$, then since this graph has the same degree sequences as $R$, we may delete the edges in the copy of $S$ and replace them with edges from a copy of $R$ so that the modified graph has the same degree sequence as $G$ but induces $R$. This is a contradiction, so $G$ is $S$-free, and similarly $G$ is $\overline{S}$-free as well.

(c) $\Rightarrow$ (a): Note that in Theorem~\ref{thm: Barrus uni}, each of the forbidden subgraphs for hereditary unigraphs induces $2K_2$, $C_4$, $R$, $\overline{R}$, $S$, or $\overline{S}$. Hence if $G$ is $\{2K_2,C_4,R,\overline{R},S,\overline{S}\}$-free, then it is a unigraph, and (a) follows trivially.

(b) $\Rightarrow$ (c): Since $G$ is pseudo-split, it is $\{2K_2,C_4\}$-free. Since each canonical component of $G$ is either chair-free or kite-free, and each of $R$, $\overline{R}$, $S$, and $\overline{S}$ induces both the chair and the kite, none of these can be induced within a canonical component of $G$. Observe now by Theorem~\ref{thm: A4 indecomp} that $R,\overline{R},S,\overline{S}$ are all indecomposable, and hence if $G$ were to contain any of these as an induced subgraph, then it would so within one of its canonical components, a contradiction.

(c) $\Rightarrow$ (b): $G$ is pseudo-split, by assumption. If $G$ has a nonsplit canonical component, then by Corollary~\ref{cor: split,pseudo-split chars} this component is $C_5$, which is both chair-free and kite-free. Let $(Q,A,B)$ be a splitted canonical component of $G$ that induces both the chair and the kite, with vertices as labeled in Figure~\ref{fig: chairkite both}. We do not initially assume that the vertex sets of the two induced subgraphs are disjoint, and for clarity in the figure, we omit the rest of the edges having an endpoint in the clique $B$. We present the rest of the proof through a series of facts.
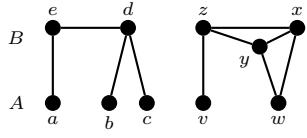
\begin{figure}
\centering
\begin{pspicture}(10,-0.6)(13.25,1)
\uput[l](9.8,0.875){\fontsize{7pt}{7pt}$B$}
\uput[l](9.8,0){\fontsize{7pt}{7pt}$A$}
\cnode*(10,0){3pt}{a} \cnode*(10,1){3pt}{e} \cnode*(11,1){3pt}{d} \cnode*(10.75,0){3pt}{b} \cnode*(11.25,0){3pt}{c}
\uput[d](10,0){\fontsize{7pt}{7pt}$a$} \uput[u](10,1){\fontsize{7pt}{7pt}$e$} \uput[u](11,1){\fontsize{7pt}{7pt}$d$} \uput[d](10.75,0){\fontsize{7pt}{7pt}$b$} \uput[d](11.25,0){\fontsize{7pt}{7pt}$c$}
\ncline{-}{a}{e} \ncline{-}{e}{d} \ncline{-}{d}{b} \ncline{-}{d}{c}
\cnode*(12,0){3pt}{v} \cnode*(12,1){3pt}{z} \cnode*(13,0){3pt}{w} \cnode*(12.75,0.75){3pt}{y} \cnode*(13.25,1){3pt}{x}
\uput[d](12,0){\fontsize{7pt}{7pt}$v$} \uput[u](12,1){\fontsize{7pt}{7pt}$z$} \uput[d](13,0){\fontsize{7pt}{7pt}$w$} \uput[dl](12.75,0.75){\fontsize{7pt}{7pt}$y$} \uput[u](13.25,1){\fontsize{7pt}{7pt}$x$}
\ncline{-}{v}{z} \ncline{-}{z}{y} \ncline{-}{z}{x} \ncline{-}{y}{x} \ncline{-}{y}{w} \ncline{-}{x}{w}
\end{pspicture}
\caption{The chair and kite from Theorem~\ref{thm: forcibly expandable}.}
\label{fig: chairkite both}
\end{figure}

\textsc{Fact 1:} If a vertex in $B \setminus \{d,e\}$ has a neighbor among $\{a,b,c\}$, then it is adjacent to to all three of these vertices. Likewise, if a vertex in $A \setminus \{v,w\}$ has a neighbor among $\{x,y,z\}$, then it is adjacent to all three of these vertices.

\emph{Proof.} Let $t$ be a vertex of $B \setminus\{d,e\}$ adjacent to at least one vertex of $\{a,b,c\}$. Let $Q'=Q[\{a,b,c,d,e,t\}]$. One can verify that unless $t$ is adjacent to all of $\{a,b,c\}$, then $Q'$ is isomorphic to one of $R$, $\overline{R}$, $S$, or $\overline{S}$. Note now that $\{2K_2,C_4,R,\overline{R},S,\overline{S}\}$ is a self-complementary set of graphs, and in $(\overline{Q},B,A)$ the vertex sets formerly inducing the chair and the kite now induce the kite and the chair, respectively. This symmetry (which we will refer to as ``complement symmetry'' in the paragraphs that follow) and the arguments above imply that in $Q$ any vertex in $A \setminus \{v,w\}$ is adjacent to either all or none of $\{x,y,z\}$.

\textsc{Fact 2:} The intersections $\{a,b,c\} \cap \{v,w\}$ and $\{d,e\} \cap \{x,y,z\}$ each contain at most one vertex.

\emph{Proof.} Every vertex in $\{a,b,c\} \setminus \{v,w\}$ has a neighbor and a nonneighbor in $\{d,e\}$; by Fact 1, we cannot have $\{d,e\} \subseteq \{x,y,z\}$. By complement symmetry, this implies that $\{v,w\} \nsubseteq \{a,b,c\}$.

\textsc{Fact 3:} The sets $\{a,b,c,d,e\}$ and $\{v,w,x,y,z\}$ are disjoint.

\emph{Proof.} Suppose first that $\{d,e\} \cap \{x,y,z\} = \{u\}$ for some $u$. If $u$ has both a neighbor $r$ and a nonneighbor $s$ in $\{a,b,c\}\setminus\{v,w\}$, then by Fact 1, $r$ is adjacent to all of $\{x,y,z\}$, and $s$ is adjacent to none of these vertices. However, this is a contradiction, since Fact 1 implies that any vertex in $\{x,y,z\}\setminus\{u\}$ must be adjacent to both or neither of $r$ and $s$. Thus $u$ cannot have both a neighbor and a nonneighbor in $\{a,b,c\} \setminus\{v,w\}$, and it follows that $a \in \{v,w\}$. Hence $\{a,b,c\} \cap \{v,w\}$ is nonempty, and by complement symmetry we conclude that $z \in \{d,e\}$.

If $a = v$, then since $v$ is adjacent to $z$ we have $z = e$. Since $b$ and $c$ are not adjacent to $e$, Fact 1 implies that $b$ and $c$ are not adjacent to $x$ or $y$. If $dw \notin E(Q)$, then $Q[\{b,c,d,w,x,y\}] \cong S$, and if $dw \in E(Q)$, then $Q[\{b,d,e,v,w,x\}] \cong R$, both contradictions. Hence $a \neq v$.

Thus $a = w$, and since $w$ is not adjacent to $z$, we have $z = d$. Since $b$ and $c$ are adjacent to $d$, Fact 1 implies that $b$ and $c$ are adjacent to both $x$ and $y$. If $ve \notin E(Q)$, then $Q[\{b,d,e,v,w,x\}] \cong \overline{R}$, and if $ve \in E(Q)$, then $Q[\{b,c,e,v,x,y\}] \cong \overline{S}$. These contradictions imply that $\{d,e\} \cap \{x,y,z\} = \emptyset$. By complement symmetry we find that $\{a,b,c\} \cap \{v,w\}=\emptyset$, so in fact $\{a,b,c,d,e\}$ and $\{v,w,x,y,z\}$ are disjoint.

\textsc{Fact 4:} For any induced 4-vertex path $P$ having a nonempty intersection with an induced chair $C$ in $Q$, there is an induced chair in $Q$ containing all vertices in $V(P)\setminus V(C)$.

\emph{Proof.} Let $p,q,r,s$ be the vertices of the path $P$, in order, so $p,s \in A$ and $q,r \in B$. For convenience, without loss of generality we will assume that $C$ is the chair with vertex set $\{a,b,c,d,e\}$. We first show that if $p \in \{a,b,c\}$, then we may assume that $q \in \{d,e\}$. If not, then by Fact 1, $q$ is adjacent to all of $\{a,b,c\}$, so $s \notin \{a,b,c\}$. Now if $r$ has two nonneighbors in $\{a,b,c\}$, then $Q$ contains an induced chair with vertices $q$, $r$, $s$, and the two nonneighbors of $r$. Otherwise, $r$ has two neighbors and one nonneighbor (namely, $p$) in $\{a,b,c\}$, so Fact 1 implies $r=d$ and hence $p=a$. Since $Q[\{a,b,d,e,q,s\}] \ncong \overline{R}$, we see that $s$ is adjacent to $e$. Then $Q[\{b,c,e,q,s\}]$ is a chair, as claimed in the fact. By a symmetric argument, we may also assume that if $s \in \{a,b,c\}$, then $r \in \{d,e\}$.

If $\{q,r\}=\{d,e\}$, then $Q$ induces a chair on the vertices $d$, $e$, $p$, $s$ and any vertex in $\{a,b,c\}\setminus\{p,s\}$. By symmetry we may assume that $q \notin \{d,e\}$. By the previous paragraph, $p \notin \{a,b,c\}$, and since $\{p,q,r,s\} \cap \{a,b,c,d,e\}$ is nonempty, $r \in \{d,e\}$. If either $q$ or $r$ has a neighbor in $\{a,b,c\}\setminus\{s\}$ that the other doesn't, then that vertex induces a chair with $\{p,q,r,s\}$. By Fact 1, $q$ is adjacent to all or none of $\{a,b,c\}$; since $r$ has both a neighbor and a nonneighbor in $\{a,b,c\}$, the only way to avoid a chair is to have $\{a,b,c\}\setminus\{s\}=\{b,c\}$, so $s=a$ and hence $r=e$. By Fact 1, $q$ is not adjacent to $b$ or $c$. Since $G[\{a,b,d,e,p,q\}] \ncong R$, $p$ is not adjacent to $d$. We then have a chair with vertex set $\{b,c,d,p,q\}$, as claimed.

\textsc{Fact 5:} $Q$ is chair-free or kite-free.

\emph{Proof.} Since $Q$ is a canonically indecomposable split graph, and no split graph can contain an induced $2K_2$ or $C_4$, Theorem~\ref{thm: A4 indecomp} implies that for any pair $t,u$ of vertices in $Q$, there is a sequence $H_1,\dots,H_\ell$ of induced 4-vertex paths such that $t \in V(H_1)$, $u \in V(H_\ell)$, and $V(H_i) \cap V(H_{i+1})$ is nonempty for $i=1,\dots,\ell-1$. Suppose that $t \in \{a,b,c,d,e\}$ and $u \in \{v,w,x,y,z\}$. Inductively applying Fact 4 to each of the paths $H_i$, we conclude that $u$ belongs to an induced chair. However this contradicts Fact 3, which implies that no induced chair and kite can have intersecting vertex sets.
\end{proof}

\section{Hereditary unigraphs} \label{sec: four}
Our results on expansions in spiders now lead us to a structural characterization of hereditary unigraphs.

In~\cite{Tyshkevich00}, Tyshkevich used the canonical decomposition to describe the structure of arbitrary unigraphs. We recall here the list of non-split indecomposable unigraphs. Let $U_s$ denote the graph formed by merging together one vertex from each connected component of $C_4 + sK_3$; note that $U_s$ is the only graph with degree sequence $(s+2,2^{2s+3})$.

\begin{thm}[\cite{Tyshkevich00}] \label{thm: Tysh uni} \mbox{}
\begin{enumerate}
\item[\textup{(a)}] A graph is a unigraph if and only if each of its canonical components is.
\item[\textup{(b)}] An indecomposable non-split graph $G$ is a unigraph if and only if $G$ or $\overline{G}$ is one of the following: $C_5$, $rK_2$ for $r \geq 2$, $K_{1,r}+sK_2$ for $r \geq 2$ and $s \geq 1$, or $U_s$ for $s \geq 1$.
\end{enumerate}
\end{thm}

Note that both the 4-pan and the co-4-pan have degree sequence $(3,2,2,2,1)$, and both are induced in $U_s$ for all $s \geq 1$.

The forbidden subgraph characterization of hereditary unigraphs allows us to adapt statement (a) of Theorem~\ref{thm: Tysh uni} to hereditary unigraphs. Let $\F$ denote the set of forbidden subgraphs in Theorem~\ref{thm: Barrus uni}.

\begin{lem} \label{lem: her uni iff components are}
A graph is a hereditary unigraph if and only if each of its canonical components is.
\end{lem}
\begin{proof}
If $G$ has a canonical component that is not a hereditary unigraph, then that component (and hence $G$) contains an induced subgraph from $\F$; thus $G$ is not a hereditary unigraph.

Conversely, if $G$ is not a hereditary unigraph, then it contains some element of $\F$ as an induced subgraph $H$. Applying Theorem~\ref{thm: A4 indecomp}, we see that each element of $\F$ is indecomposable. Hence $H$ must be induced in a single canonical component of $G$, preventing that component from being a hereditary unigraph.
\end{proof}

Thus to characterize the structure of hereditary unigraphs it suffices to characterize the canonical components that are $\F$-free. When a canonical component is non-split, Theorem~\ref{thm: Tysh uni} yields an immediate result.

\begin{cor} \label{cor: indecomp non-split her uni}
An indecomposable non-split graph $G$ is a hereditary unigraph if and only if $G$ or $\overline{G}$ is one of the following: $C_5$, $rK_2$ for $r \geq 2$, or $K_{1,r}+sK_2$ for $r \geq 2$ and $s \geq 1$.
\end{cor}

Of the forbidden subgraphs listed in Theorem~\ref{thm: Barrus uni}, only $R$, $\overline{R}$, $S$, and $\overline{S}$ are split. Theorems~\ref{thm: Barrus uni}, \ref{thm: split forb sub}, \ref{thm: chair-free structure}, and~\ref{thm: forcibly expandable} and Corollary~\ref{cor: kite-free structure} then yield the following:

\begin{cor}\label{cor: indecomp split her uni}
The following are equivalent for an indecomposable split graph $G$:
\begin{enumerate}
\item[\textup{(a)}] $G$ is a hereditary unigraph;
\item[\textup{(b)}] $G$ is $\{R,\overline{R},S,\overline{S}\}$-free;
\item[\textup{(c)}] $G$ is chair-free or kite-free;
\item[\textup{(d)}] $G$ is isomorphic to $K_1$ or is a top-expanded spider or a bottom-expanded spider.
\end{enumerate}
\end{cor}

Pulling these results together yields a complete structural characterization of hereditary unigraphs:

\begin{thm}\label{thm: main structural}
A graph $G$ with canonical decomposition $G=(G_k,A_k,B_k)\circ \dots \circ (G_1,A_1,B_1) \circ G_0$ is a hereditary unigraph if and only if each of its split canonical components with more than one vertex is a top-expanded spider or a bottom-expanded spider, and $G_0$ or $\overline{G_0}$ is either split or isomorphic to $C_5$, $rK_2$ for $r \geq 2$, or $K_{1,r}+sK_2$ for $r \geq 2$ and $s \geq 1$.
\end{thm}

We return briefly to the example families of hereditary unigraphs listed in~\eqref{eq: tower}. Threshold graphs are precisely the $\{2K_2,C_4,P_4\}$-free graphs~\cite{ChvatalHammer73}. By Theorem~\ref{thm: A4 indecomp}, these graphs (including the complete and edgeless graphs) have canonical decompositions where the indecomposable components each contain a single vertex. Though matrogenic graphs and matroidal graphs were defined separately, in terms of matroids in slightly different contexts, these classes' forbidden subgraph characterizations~\cite{FoldesHammer78,Peled77} show that matroidal graphs are precisely the $C_5$-free matrogenic graphs. Based on the work of F\"{o}ldes and Hammer and Peled, as well as results of Marchioro et al.~\cite{MarchioroEtAl84} and Tyshkevich~\cite{Tyshkevich84}, we recall the following characterization of the structure of matrogenic graphs. A \emph{net} is a headless prime spider satisfying condition iii(a) in Definition~\ref{defn: prime spider}; its complement is a headless prime spider satisfying condition iii(b).

\begin{thm}
A graph $G$ with canonical decomposition $G=(G_k,A_k,B_k)\circ \dots \circ (G_1,A_1,B_1) \circ G_0$ is matrogenic if each split canonical component with more than one vertex is a net or the complement of a net, and $G_0$ is split or isomorphic to one of $C_5$, $rK_2$, or $\overline{rK_2}$ for $r \geq 2$.
\end{thm}

Note that the canonical decompositions of complete, edgeless, threshold, matroidal, and matrogenic graphs clearly satisfy the hypotheses of Theorem~\ref{thm: main structural}. In particular, matrogenic graphs are simply hereditary unigraphs where the possibilities for $G_0$ are restricted, and the canonical components with more than one vertex are simply headless prime spiders (rather than having been obtained from prime spiders by substituting cliques or independent sets of size at least 2 for body vertices or feet).

\section{Erd\H{o}s--Gallai equalities and the canonical decomposition} \label{sec: five}
Having characterized the structure of hereditary unigraphs in Theorem~\ref{thm: main structural}, in the remainder of the paper we turn to the degree sequences of these graphs. In preparation for our main result (Theorem~\ref{thm: deg seq for her unis, part II}), in this section we describe the relationship between the canonical decomposition and the Erd\H{o}s--Gallai inequalities. For convenience, we abbreviate the latter as \emph{EG-inequalities}.

As mentioned in the introduction, threshold graphs and split graphs have degree sequence characterizations in terms of equalities among the EG-inequalities. Given a degree sequence $d=(d_1,\dots,d_n)$, let $EG(d)$ denote the list of nonnegative integers $k$ for which the $k$th EG-inequality holds with equality, ordered from smallest to largest. We adopt the convention that when the starting value of a sum's index exceeds the ending value, the value of the sum is 0; hence $EG(d)$ begins with $0$ for every $d$. Henceforth, let $m(d)=\max\{i:d_i \geq i-1\}$.

\begin{thm} \label{thm: deg seqs for split, pseudo}
Let $G$ be a graph with degree sequence $d=(d_1,\dots,d_n)$, and let $m=m(d)$.
\begin{enumerate}
\item[\textup{(a)}] (Hammer et al.~\cite{HammerEtAl78}) $G$ is threshold if and only if \[\sum_{i=1}^k d_i = k(k-1) + \sum_{i=k+1}^n \min\{k, d_i\}\] for $k=1,\dots,m(d)$.
\item[\textup{(b)}] (Hammer and Simeone~\cite{HammerSimeone81}; Tyshkevich et al.~\cite{Tyshkevich80,TyshkevichEtAl81}) $G$ is split if and only if \[\sum_{i=1}^m d_i = m(m-1)+\sum_{i=m+1}^n d_i.\]
\item[\textup{(c)}] (Maffray and Preissmann~\cite{MaffrayPreissmann94}) $G$ is pseudo-split if and only if $G$ is split or \[\sum_{i=1}^q d_i = q(q+4)+\sum_{i=q+6}^n d_i\] and \[d_{q+1}=d_{q+2}=d_{q+3}=d_{q+4}=d_{q+5}=q+2,\] where $q=\max\{i:d_i \geq i+4\}$.
\end{enumerate}
\end{thm}

Thus $d$ is the degree sequence of a split graph if and only if $m(d)$ is a term of $EG(d)$. Furthermore, as we will note later in Corollary~\ref{cor: k leq m}, if the $k$th EG-inequality holds with equality, then $k \leq m(d)$, so $d$ is the degree sequence of a threshold graph if and only if $EG(d)=(0,1,\dots,m(d))$. The first equality in Theorem~\ref{thm: deg seqs for split, pseudo}(c) is equivalent, in light of the other hypothesis, to the requirement that $EG(d)$ contain $q$ as a term. The form of the characterization in (c), i.e., requirements on both $EG(d)$ and on specific terms of the degree sequence, will appear in our characterization of the degree sequences of hereditary unigraphs in Section~\ref{sec: six}.

All threshold graphs having at least two vertices are canonically decomposable, as are pseudo-split graphs inducing $C_5$ (other than $C_5$ itself). We now show that, as with the threshold and split graphs, canonically decomposable graphs can be recognized from their degree sequences by computing $EG(d)$. Before giving our result, we note that Tyshkevich~\cite{Tyshkevich00} provided necessary and sufficient conditions for a degree sequence to belong to a decomposable graph. The conditions closely follow an instance of the Fulkerson--Hoffman--McAndrew criteria for testing if a degree sequence is graphic. These latter criteria are equivalent to the EG-inequalities for sequences of nonnegative integers with even sum (Chapter 3 of~\cite{MahadevPeled95} discusses these and several other such criteria and proves their collective equivalence). Our purpose for defining and using $EG(d)$ is to make clear how our characterization of hereditary unigraphs in Section~\ref{sec: six} relates to the characterizations in Theorem~\ref{thm: deg seqs for split, pseudo} and those known for other classes. Using $EG(d)$ will also allow us to present the degree sequence characterization of hereditary unigraphs in a more straightforward way, rather than forcing us to first determine the degree sequences of the canonical components of a graph.

We begin with a few observations on the canonical decomposition.

\begin{obs}\label{obs: degree order}
Suppose that $G$ is a graph with canonical decomposition $(G_\ell,A_\ell,B_\ell) \circ \dots \circ (G_1,A_1,B_1) \circ G_0$, where $A_0$ and $B_0$ partition $V(G_0)$ into an independent set and clique, respectively, if $G_0$ is split. Consider the list \[B_\ell, \dots,B_0, A_0, \dots,A_\ell\] if $G_0$ is split, and the list \[B_\ell, \dots,B_1, V(G_0), A_1, \dots,A_\ell\] if $G_0$ is not split. Let $u$ and $v$ be vertices from distinct sets in the list. If the set containing $u$ precedes the set containing $v$ in the list, then $d_G(u)\geq d_G(v)$. The converse is true as well unless $d_G(u)=d_G(v)$, in which case the canonical components containing $u$ and $v$ are both isomorphic to $K_1$.
\end{obs}

Call a canonically indecomposable graph \emph{nontrivial} if it has more than one vertex. By Theorem~\ref{thm: A4 indecomp}, in a nontrivial indecomposable graph $H$ each vertex belongs to an induced subgraph of $H$ isomorphic to $2K_2$, $C_4$, or $P_4$. This leads to the following.

\begin{obs}\label{obs: nbr and non}
Every vertex in a nontrivial canonically indecomposable graph has both a neighbor and a non-neighbor in the graph. If $(H,A,B)$ is a splitted indecomposable graph, then each vertex of $A$ has both a neighbor and a nonneighbor in $B$, and each vertex of $B$ has both a neighbor and a nonneighbor in $A$.
\end{obs}

\begin{lem}\label{lem: EG equality}
Let $Q$ be a set of vertices of $G$, and let $\overline{Q} = V(G) - Q$. Then \[\sum_{q \in Q} d_G(q) = |Q|(|Q|-1) + \sum_{p \in \overline{Q}} \min\{|Q|, d_G(p)\}\] if and only if $G=(G[P\cup Q], P, Q) \circ G[T]$, where $P$ consists of all vertices of $G$ with degree less than $|Q|$, and $T=V(G)-P-Q$.
\end{lem}
\begin{proof}
Let $|Q,\overline{Q}|$ denote the number of edges having an endpoint in each of $Q$ and $\overline{Q}$. Observe that $\sum_{q \in Q} d_G(q) - |Q|(|Q|-1)$ is a lower bound on $|Q,\overline{Q}|$, where equality holds if and only if every pair of vertices in $Q$ is adjacent. Observe also that $\sum_{p \in \overline{Q}} \min\{|Q|,d_G(p)\}$ is an upper bound for $|Q,\overline{Q}|$, with equality holding if and only if those vertices in $\overline{Q}$ whose degree is less than $|Q|$ (i.e., vertices in $P$) only have neighbors in $Q$, and those vertices whose degree is at least $|Q|$ (i.e., vertices in $T$) are adjacent to every vertex of $Q$. Thus \[\sum_{q \in Q} d_G(q) = |Q|(|Q|-1) + \sum_{p \in \overline{Q}} \min\{|Q|, d_G(p)\}\] if and only if $Q$ is a clique whose neighbors are all adjacent to every vertex of $T$ while $P$ is an independent set whose vertices are adjacent to no vertex of $T$. This latter condition is clearly equivalent to $G=(G[P\cup Q], P, Q) \circ G[T]$.
\end{proof}

\begin{cor}\label{cor: k leq m}
Let $d$ be a graphic sequence. If $d$ satisfies the $k$th EG-inequality with equality, then $k \leq m(d)$.
\end{cor}
\begin{proof}
Let $Q$ be a set of $k$ vertices of $G$ with the largest degrees in $G$. By Lemma~\ref{lem: EG equality}, every vertex in $Q$ is adjacent to the $k-1$ other vertices of $Q$, so $m(d) \geq k$.
\end{proof}

Given a degree sequence $d$, the \emph{conjugate} sequence $d^*$ is defined by letting $d^*_j = \max\{i: d_i \geq j\}$ for $j\geq 0$. For nonnegative $j$ let $\delta_j=|\{i:i>j\text{ and }d_i=j\}|$.

\begin{thm}~\label{thm: deg seqs and canon comps}
Let $G$ be a graph with degree sequence $d=(d_1,\dots,d_n)$ and vertex set $\{v_1,\dots,v_n\}$, indexed so that $d_G(v_i)=d_i$. Suppose that $(G_\ell,A_\ell,B_\ell) \circ \dots \circ (G_1,A_1,B_1) \circ G_0$ is the canonical decomposition of $G$, where $A_0$ and $B_0$ partition $V(G_0)$ into an independent set and a clique, respectively, if $G_0$ is split.
\begin{enumerate}
\item[\textup{(a)}] A nonempty set $W \subseteq V(G)$ is equal to the clique $B_j$ in the canonical component $(G_j,A_j,B_j)$ if and only if $W=\{v_i : k<i \leq k'\}$ for a pair $k,k'$ of consecutive terms in $EG(d)$. In this case the corresponding independent set $A_j$ is precisely the set $\{v \in V(G): k<d_G(v)<k'\}$.
\item[\textup{(b)}] Given a term $k$ of $EG(d)$, if $i>k$ and $d_G(v_i)=k$ then the canonical component containing $v_i$ is trivial.
\item[\textup{(c)}] If $k$ is a positive term of $EG(d)$ and $(G_j,A_j,B_j)$ is the splitted canonical component containing $v_k$, then \[d^*_k=\left|V(G_0) \cup \bigcup_{i=1}^\ell B_i \cup \bigcup_{i=1}^{j-1} A_i\right|.\]
\item[\textup{(d)}] Let $(G_j,A_j,B_j)$ be a splitted canonical component, and let $k$ and $k'$ be the consecutive terms of $EG(d)$ such that $B_j = \{v_i:k<i\leq k'\}$. If $v \in B_j$, then $d_G(v) = d_{G_j}(v)-k'+k+d^*_{k'}.$ If $v \in A_j$, then $d_G(v) =d_{G_j}(v)+k$.
\end{enumerate}
\end{thm}
\begin{proof}
(a): We prove first that a vertex $v_s$ belongs to $\bigcup_{i\geq 0} B_i$ if and only if $s \leq t$, where $t$ is the last term of $EG(d)$. Indeed, applying Lemma~\ref{lem: EG equality} with $Q=\bigcup_{i\geq 0}^\ell B_i$, we see that $d$ satisfies the $|Q|$th EG-inequality, so $s \leq |Q| \leq t$. Conversely, if $s \leq t$, then $v_s$ belongs to a set $Q'$ of $t$ vertices with the highest degrees in $G$. Lemma~\ref{lem: EG equality} implies the existence of subsets $P'$ and $T'$ of $V(G)$ such that $G=(G[P'\cup Q'], P', Q') \circ G[T']$, where $G[T']$ is either empty or canonically indecomposable. By Theorem~\ref{thm: deg seqs for split, pseudo}(b) and the uniqueness of the canonical decomposition, $Q'=\bigcup_{i \geq 0}^\ell B_i$.

Now let $W$ be a nonempty subset of $V(G)$, and suppose $v_s \in W$. If $s > t$, then by the last paragraph $v_s$ belongs to no clique $B_i$ and clearly belongs to no set $\{v_i : k<i \leq k'\}$ for $k,k' \in EG(d)$, so $W$ can equal neither of these. Assume that $s \leq t$, and let $B_j$ be the set in $\bigcup_{i\geq 0} B_i$ containing $v_s$. By Observation~\ref{obs: degree order}, we may assume that $B_j = \{v_a,v_{a+1},\dots,v_z\}$ for some $a,z \leq t$. Let $k,k'$ be the consecutive terms in $EG(d)$ such that $k<s \leq k'$. We show now that $a=k+1$ and $z=k'$.

By Lemma~\ref{lem: EG equality}, we may write $G=(G[P\cup Q], P, Q) \circ G[T]$, where $Q=\{v_1,\dots,v_{k}\}$, $P$ consists of all vertices of $G$ with degree less than $k$, and $T=V(G)-P-Q$, so $k+1 \leq a$. On the other hand, since we may write $G=(G[A \cup B], A, B)\circ G[C]$ for subsets $A$, $B$, and $C$ of $V(G)$ such that $B=\{v_1,\dots,v_{a-1}\}$, Lemma~\ref{lem: EG equality} implies that $a-1 \leq k$. Hence $a=k+1$, and similar applications of Lemma~\ref{lem: EG equality} show that $z=k'$. We conclude that $W=B_j$ if and only if $W=\{v \in V(G): k<d_G(v)<k'\}$.

Now let $A' = \{v \in V(G): k<d_G(v)<k'\}$. By Observation~\ref{obs: nbr and non} and the requirements of the canonical decomposition, vertices in $A_j$ have degree strictly between $k$ and $k'$, so $A_j \subseteq A'$. These same results also imply that vertices in $\bigcup_{i>j} B_i$, $B_j$, and $\bigcup_{i <j} V(G_i)$ all have degree at least $k'$, while vertices in $\bigcup_{i>j} A_i$ have degree at most $k$. Thus $A_j = A'$.

(b): Let $k$ be a term of $EG(d)$. Note that each isolated vertex of $G$ is a trivial component in the canonical decomposition (if $G$ has $t$ isolated vertices, then the canonical components $G_i$ with $\ell-t+1\leq i \leq \ell$ are each splitted with their respective vertex sets satisfying $B_i =\emptyset$ and $|A_i|=1$). We may therefore assume that $k>0$. Let $(G_j,A_j,B_j)$ be the splitted canonical component containing $v_k$. Further let $W=\{v_i \in V(G) : i>k\text{ and }d_i=k\}$, and suppose that $G_s$ is a canonical component of $G$ containing a vertex $v$ of $W$. Observation~\ref{obs: nbr and non} and the requirements of the canonical decomposition imply that vertices in $\bigcup_{i\geq j} A_i$ have degree less than $k$. By the arguments of part (a) above, if some nontrivial canonical component $G_{j'}$ of $G$ satisfies $j'<j$, then vertices in $V(G_{j'})$ have degree greater than $k$. Thus $G_s$ is a trivial component.

(c): By definition, $d^*_k$ is the number of terms of $d$ having value at least $k$. By (a) and Observation~\ref{obs: degree order}, since $G_j$ is the canonical component containing $v_k$, the set of vertices having degree at least $k$ is $V(G_0) \cup \bigcup_{i=1}^\ell B_i \cup \bigcup_{i=1}^{j-1} A_i$.

(d): If $v \in B_j$, the $v$'s neighborhood in $G$ consists of its neighbors in $A_j$ and all vertices of $V(G_0) \cup \bigcup_{i=1}^\ell B_i \cup \bigcup_{i=1}^{j-1} A_i$ other than itself. Since  $k'-k-1$ of these latter vertices are vertices in $B_j$ to which $v$ is adjacent, (c) yields the expression for $d_G(v)$ immediately. If $v \in A_j$, then the neighborhood of $v$ in $B$ consists of the neighbors of $v$ in $G_j$ and the vertices $\{v_1,\dots,v_k\}$, so $d_G(v)=d_{G_j}(v)+k$.
\end{proof}

\section{Degree sequences of hereditary unigraphs} \label{sec: six}

In this section we use Theorem~\ref{thm: deg seqs and canon comps} and the structural results in Section~\ref{sec: four} to characterize the degree sequences of hereditary unigraphs.

\begin{lem}\label{lem: unique split ptn}
If $G$ is split with degree sequence $d=(d_1,\dots,d_n)$ and $m=m(d)$, then the vertices with degree at least $d_m$ form the clique of a partition of $V(G)$ into a clique and an independent set. If $G$ is indecomposable and has more than one vertex, then this is the unique such partition.
\end{lem}
\begin{proof}
The first statement is proved by Hammer and Simeone in~\cite{HammerSimeone81}. If $G$ is indecomposable and has more than one vertex, then by Theorems~\ref{thm: split forb sub} and~\ref{thm: A4 indecomp} each vertex of $G$ belongs to an induced $P_4$. Since the vertex set of $P_4$ has a unique partition into an independent set and a clique, the same property follows for the vertex set of $G$.
\end{proof}

\begin{thm}\label{thm: deg seq for her unis, part I}
A graph $G$ is a hereditary unigraph if and only if for each canonical component $G'$ of $G$, the degree sequence $d'=(d'_1,\dots,d'_p)$ of $G'$ satisfies the following:
\begin{enumerate}
\item[\textup{(i)}] if $G'$ is split then $p=1$ or $d'_1=\dots=d'_{m'} \in \{m', p-2\}$ or $d'_{m'+1}=\dots=d'_p \in \{1, m'-1\}$, where $m' = m(d')$;
\item[\textup{(ii)}] if $G'$ is not split then $d'$ is one of the sequences \[(2^5), \quad (1^{2r}), \quad (r,1^{2s+r}), \quad ((2r-2)^{2r}), \quad ((2s+r-1)^{2s+r},2s)\] for some $r \geq 2$ and $s \geq 1$.
\end{enumerate}
\end{thm}

\begin{proof}
Theorem~\ref{thm: main structural} characterizes hereditary unigraphs in terms of their canonical components. Let $G$ be an arbitrary graph, and let $G'$ be a canonical component of $G$ with degree sequence $d'$. It suffices to show that $G'$ has one of the forms listed in Theorem~\ref{thm: main structural} if and only if its degree sequence meets the conditions (i) and (ii) here.

Suppose first that $G'$ is split. We note that $G'$ has one vertex if and only if $p=1$ and assume that $G'$ has more than one vertex. Let $\{v'_1,\dots,v'_p\}$ be the vertex set of $G'$, indexed so that $v'_i$ has degree $d'_i$ for each $i$. By Lemma~\ref{lem: unique split ptn}, $B'= \{v'_1,\dots,v'_{m'}\}$ is the set of clique vertices in the unique partition of $V(G')$ into a clique and independent set, and $A'=\{v_{m'+1},\dots,v'_p\}$ is the independent set.

If $G'$ is a top-expanded spider, then every two vertices of $B'$ have the same degree. Depending on which case holds in Definition~\ref{defn: prime spider}(iii), each vertex in $B'$ is adjacent to either $1$ or $|A'|-1$ vertices of $A'$, along with all other vertices of $B'$. Hence $d'_1=\dots=d'_{m'} \in \{m', p-2\}$. Conversely, suppose $d'_1=\dots=d'_{m'} \in \{m', p-2\}$. Every vertex $v'_j$ of $B'$ is adjacent to the $m'-1$ other vertices in $B'$. Thus if $d'_j=m'$, then $v'_j$ has exactly one neighbor in $A'$. In this case, we may partition the vertices of $B'$ into classes $C_1,\dots,C_t$ according to their neighbors in $A'$. Since $G'$ is indecomposable, every vertex in $A'$ has at least one neighbor in $B'$ by Observation~\ref{obs: nbr and non}. It follows that $G'$ is a top-expanded spider, where $C_1,\dots,C_t$ are the cliques substituted in forming $G'$ from a prime spider. Similarly, if $d'_j = p-2$, then $v'_j$ has exactly one nonneighbor in $A'$, and we may partition $B'$ into $C_1,\dots,C_t$ according to their nonneighbors in $A'$. Since $G$ is indecomposable, each vertex in $A'$ has at least one nonneighbor in $B'$, so once again $G'$ is a top-expanded spider obtained by substituting $C_1,\dots,C_t$ for body vertices in a prime spider.

Hence $G'$ is a top-expanded spider if and only if $d'_1=\dots=d'_{m'} \in \{m', p-2\}$. Similarly arguments show that $G'$ is a bottom-expanded spider if and only if $d'_{m'}=\dots=d'_p \in \{1,m'-1\}$.

Now suppose that $G'$ is not split. For $r \geq 2$ and $s \geq 1$, each of $C_5$, $rK_2$, and $K_{1,r}+sK_2$ is clearly a unigraph with degree sequence $(2^5)$, $(1^{2r})$, $(r,1^{2s+r})$, respectively. The complements $\overline{rK_2}$ and $\overline{K_{1,r}+sK_2}$ are also unigraphs, having degree sequences $((2r-2)^{2r})$ and $((2s+r-1)^{2s+r},2s)$, respectively.
\end{proof}

\begin{thm}\label{thm: deg seq for her unis, part II}
If $G$ is a graph with degree sequence $d=(d_1,\dots,d_n)$, then $G$ is a hereditary unigraph if and only if the following conditions hold:
\begin{enumerate}
\item[\textup{(a)}] For every pair $k,k'$ of consecutive terms in $EG(d)$ such that $k' \geq k+2$, either
\begin{enumerate}
\item[\textup{(i)}] the terms $d_i$ with $k <i \leq k'$ are all equal and belong to $\{d^*_{k'}, d^*_{k}-\delta_{k}-2\}$, or
\item[\textup{(ii)}] the terms $d_i$ such that $k<d_i<k'$ are all equal and belong to $\{k+1,k'-1\}$.
\end{enumerate}
\item[\textup{(b)}] If $t$ denotes the last term of $EG(d)$, then if any terms $d_i$ satisfy $i>t$ and $d_i>t$, they collectively form one of the sequences
\begin{multline}\nonumber ((t+2)^5), \quad ((t+1)^{2r}), \quad (t+r, (t+1)^{2s+r}),\\ ((t+2r-2)^{2r}),\quad ((t+2s+r-1)^{2s+r}, t+2s)\end{multline} for some $r \geq 2$ and $s \geq 1$.
\end{enumerate}
\end{thm}

\begin{proof}
We show that conditions (a) and (b) are equivalent to parts (i) and (ii), respectively, of Theorem~\ref{thm: deg seq for her unis, part I}. Let $G$ be an arbitrary graph, and denote the vertices of $G$ as $\{v_1,\dots,v_n\}$ so that $v_i$ has degree $d_i$ for each $i$.

Let $k,k'$ be a pair of consecutive terms in $EG(d)$ such that $k' \geq k+2$. By Theorem~\ref{thm: deg seqs and canon comps}(a), the triple $(G[T\cup U],T,U)$ is a canonical component of $G$, where $T=\{v \in V(G) : k<d_G(v)<k'\}$ and $U = \{v_i : k<i<k'\}$. Suppose that $G[T\cup U] = G_j$, so that $T=A_j$ and $U=B_j$.

Suppose (a) holds. In case (i) the vertices in $U$ all have the same degree in $G_j$, and this degree is $|B_j|$ or $|V(G_j)|-2$. Theorem~\ref{thm: deg seqs and canon comps} implies that the degree in $G$ of a vertex in $U$ is either $|B_j|-(k'-k)+d^*_{k'}=d^*_{k'}$ or $|A_j\cup B_j|-(k'-k)+d^*_{k'}-2 = d^*_{k}-\delta_{k}-2$. In case (ii) the vertices in $T$ all have the same degree, which is $1$ or $|U|-1$. Theorem~\ref{thm: deg seqs and canon comps} implies that the degree of a vertex in $T$ is either $k+1$ or $k+|U|-1 = k'-1$.

Conversely, let $d'=(d'_1,\dots,d'_p)$ be the degree sequence of $G_j$. Lemma~\ref{lem: unique split ptn} implies that $m(d')=|B_j|$. Suppose first that the vertices in $B_j$ all have the same degree in $G$, which is either $\{d^*_{k'}$ or $d^*_{k}-\delta_{k}-2\}$. By Theorem~\ref{thm: deg seqs and canon comps}, the degree of these vertices in $G_j$ is either $d^*_{k'}+k'-k-d^*_{k'} = m(d')$ or $d^*_{k}-\delta_{k}-d^*_{k'}+ k' -k-2 = |A_j\cup B_j|-2 = p-2$. If instead the vertices of $A_j$ all have the same degree in $G$, which is either $k+1$ or $k'-1$, then by Theorem~\ref{thm: deg seqs and canon comps}, the degree of these vertices in $G_j$ is either $1$ or $k'-1-k = m(d')-1$. Hence (a) is satisfied.

Now let $d'$ denote the list of terms $d_i$ of $d$ such that $i>t$ and $d_i>t$, where $t$ is the last term of $EG(d)$. It follows from Theorem~\ref{thm: deg seqs and canon comps} and the adjacency requirements of the canonical decomposition that the vertices of $G$ corresponding to the terms in $d'$ comprise $V(G_0)$, and $G_0$ is not split. Each vertex in $V(G_0)$ has a degree in $G$ that is clearly $t$ units larger than the its degree in $G_0$, so $d'$ has one of the forms listed above in (b) if and only if the degree sequence of $G_0$ is one listed in (ii) in Theorem~\ref{thm: deg seq for her unis, part I}.
\end{proof}

Let $G$ be a graph with $n$ vertices. As stated in~\cite{Tyshkevich00}, the degree sequences of the canonical components of $G$ may be computed from the degree sequence of $G$ in linear time. (Indeed, one may use Theorem~\ref{thm: deg seqs and canon comps} to modify a linear-time algorithm for verifying the Erd\"{o}s--Gallai inequalities, such as the one in~\cite{IvanyiEtAl11}, to derive the components' degree sequences.) In the process of decomposition one may keep track of which components are split. Checking then that the degree of each vertex in its respective canonical component satisfies the conditions in Theorem~\ref{thm: deg seq for her unis, part I} can be accomplished in $O(n)$ steps. We thus have the following.

\begin{thm}
Hereditary unigraphs may be recognized from their degree sequences in linear time.
\end{thm}

Degree sequences of general unigraphs were studied by several authors (see~\cite{KleitmanLi75}--\cite{Koren76b}, \cite{Li75}, \cite{Tyshkevich00}, \cite{TyshkevichChernyak78}--\cite{TyshkevichChernyak79}). Tyshkevich~\cite{Tyshkevich00} gave a complete characterization of unigraphic degree sequences using the canonical decomposition. Theorems~\ref{thm: deg seq for her unis, part I} and \ref{thm: deg seq for her unis, part II} on hereditary unigraphs may be seen to be a special case of these previous results.

Theorems~\ref{thm: deg seq for her unis, part I} and \ref{thm: deg seq for her unis, part II} in turn generalize the degree sequence characterizations for our example families of hereditary unigraphs. For example, the reliance of the characterization in Theorem~\ref{thm: deg seq for her unis, part II} on the list $EG(d)$ resembles that of Theorem~\ref{thm: deg seqs for split, pseudo} for the threshold and split graphs. Indeed, as mentioned in Section~\ref{sec: five}, the threshold graphs are precisely those graphs for which the terms of $EG(d)$ are consecutive integers and, if $t$ is the last term of $EG(d)$, every degree sequence term appears within the first $t$ terms or has value less than $t$.

Moreover, the degree sequence characterization for matroidal and matrogenic graphs presented in~\cite{Tyshkevich84} relies on the decomposition of the degree sequence of a graph into the degree sequences of its canonical components. These component degree sequences are then checked to see if they are each the degree sequence of a single vertex, a net or net-complement, $rK_2$ for some $r\geq 2$, or $C_5$ (this last not being allowed for matroidal graphs). It is easy to see that these characterizations are special cases of Theorem~\ref{thm: deg seq for her unis, part I}. In fact, to the characterization in~\cite{Tyshkevich84} and to another characterization in~\cite{MarchioroEtAl84} of degree sequences of matrogenic graphs we may add the following: A graph $G$ is matrogenic (or matroidal) if and only if its degree sequence satisfies the conditions of Theorem~\ref{thm: deg seq for her unis, part II} with both (i) and (ii) holding in (a), and with the lists $(t+r, (t+1)^{2s+r})$ and $((t+2s+r-1)^{2s+r}, t+2s)$ (and $((t+2)^5)$, for matroidal graphs) omitted in (b).

\end{document}